\documentclass[twoside]{article}

\usepackage{amsmath,amssymb,amsfonts,amsthm,fancyhdr}          
\usepackage{epsfig,graphicx,picins,picinpar,subfigure}         
\usepackage{pstricks}                                          
\usepackage{fancyvrb}                                          
\usepackage[numbers,sort&compress]{natbib}                     

\usepackage{graphicx,amssymb,mathrsfs,amsmath,latexsym,amsfonts,amsthm,bbding}

\catcode`@=11 \long\def\@makefntext#1{\noindent #1}
\newskip\tabcentering \tabcentering=1000pt plus 1000pt minus 1000pt

\def\MCH#1#2{\setbox0=\hbox{\raise#1\hbox{#2}}\smash{\box0}}

\let\@oddfoot\@empty  \let\@evenfoot\@empty

\def\@cite#1#2{{#1\if@tempswa,#2\fi}}

\renewcommand\@biblabel[1]{[${#1}$]}

\def\@evenhead{\thepage}
\def\@oddhead{\hbox to \textwidth{\footnotesize{\it
On the Random Conjugate Spaces of a Random Locally
Convex Module} \hfill\thepage}}

\begin{document}

\renewcommand{\theequation}{\arabic{section}.\arabic{equation}}
\makeatletter      
\@addtoreset{equation}{section}
\makeatother       

\renewcommand{\thetable}{\arabic{table}}
\renewcommand{\thefigure}{\arabic{figure}}

\newtheoremstyle{mythm}{3pt}{3pt}{}{}{\bfseries}{}{2mm}{}
\theoremstyle{mythm}

\newtheorem{definition}{{\bf{Definition}}}[section]
\newtheorem{proposition}{{\bf{Proposition}}}[section]
\newtheorem{property}{{\bf{Property}}}[section]
\newtheorem{theorem}{{\bf{Theorem}}}[section]
\newtheorem{lemma}{{\bf{Lemma}}}[section]
\newtheorem{corollary}{{\bf{Corollary}}}[section]
\newtheorem{axiom}{\hspace{2em}{\bf{Axiom}}}[section]
\newtheorem{exercise}{\hspace{2em}{\bf{Exercise}}}[section]
\newtheorem{question}{\hspace{2em}{\bf{Question}}}
\newtheorem{example}{{\bf{Example}}}
\newtheorem{notation}{\hspace{2em}{\bf{Notation}}}
\newtheorem{remark}{\hspace{2em}{\bf{Remark}}}

\newcommand{\authorsinfo}{Corresponding author}

\title{On the Random Conjugate Spaces of a Random Locally
Convex Module  
\footnote{Supported by National Natural Science Foundation of China (Grant No. 10871016)}}

\author{\bf{Guo Tiexin}\footnote{\authorsinfo}, Zhao Shien\\[-1pt]
\textit{\small LMIB and School of Mathematics and Systems Science,
Beihang University,
Beijing 100191, PR.} \\[-2mm]
\textit{\small China}\\[-2pt]
\textit{\small E-mail: txguo@buaa.edu.cn}}

\date{}
\maketitle

\vspace{-6mm}

\vskip 1 mm

\noindent{\small {\small\bf Abstract}~~Theoretically speaking, there
are four kinds of possibilities to define the random conjugate
space of a random locally convex module. The purpose of this paper
is to prove that among the four kinds there are only two which are
universally suitable for the current development of the theory of
random conjugate spaces: in this process we also obtain a somewhat
surprising and crucial result that for a random normed module with
base $(\Omega,{\cal F},P)$ such that $(\Omega,{\cal F},P)$ is nonatomic
then the random normed module is a totally disconnected topological space when it is
endowed with the locally $L^{0}-$convex topology. \ \

\vspace{1mm}\baselineskip 12pt


\noindent{\small\bf MR(2000) Subject Classification}~ 46A20, 46A22, 46A16,
46A05, 46A25\ \ {\rm }}

\section{Introduction and main results}

The notion of a random locally convex module is a random generalization of that of a locally convex space.
However, the theory of classical conjugate spaces for locally convex spaces universally fails to serve for
the development of random locally convex modules, which motivated us to have developed the theory of random
conjugate spaces. Now, the theory of random conjugate spaces
has played an essential role in both the theory of random locally convex modules and the theory of conditional
risk measures, see \cite{Guo-Relations between, Guo-progress} for details.

Since there are two kinds of useful topologies for every random locally convex module, namely
the $(\varepsilon,\lambda)-$topology introduced by Guo in \cite{Guo-Survey} and the locally
$L^{0}-$convex topology introduced by Filipovi\'c, et.al in \cite{FKV}, we can naturally consider
four kinds of possibilities to define the random conjugate space of a random locally convex module.
The purpose of this paper is to further discuss the relations among the four kinds of definitions.

To introduce the main results of this paper, let us first recall some notation and terminology as follows.

Throughout this paper, $(\Omega,\mathcal {F},P)$ denotes a
probability space, $K$ the real number field $R$ or the complex number field $C$ and
$L^{0}(\mathcal {F},K)$ the algebra of equivalence classes of
$K-$valued measurable random variables on $\Omega$
under the ordinary scalar multiplication, addition and
multiplication operations on equivalence classes.

Given a random locally convex module $(E,\mathcal{P})$ over $K$ with
base $(\Omega,\mathcal {F},P)$, let
$\mathcal{T}_{\varepsilon,\lambda}$ and $\mathcal{T}_{c}$ denote the
$(\varepsilon,\lambda)-$topology and the locally $L^{0}-$convex
topology for $E$, respectively, see \cite{Guo-Relations between, FKV} and also Section 2 for the definitions of
these two kinds of topologies. $L^{0}(\mathcal{F},K)$, as a
special random locally convex module,  also has the corresponding
two kinds of topologies. The $(\varepsilon,\lambda)-$topology for
$L^{0}(\mathcal{F},K)$ is exactly the topology of convergence in
probability $P$, $(L^{0}(\mathcal{F},K),\mathcal
{T}_{\varepsilon,\lambda})$ is a topological algebra over $K$, and in particular when $(E,\mathcal{P})$ is endowed with the $(\varepsilon,\lambda)-$topology $(E,\mathcal{P})$ is a topological module over the topological algebra $(L^{0}(\mathcal{F},K),\mathcal
{T}_{\varepsilon,\lambda})$.
However, $(L^{0}(\mathcal{F},K),\mathcal{T}_{c})$ is only a topological ring, namely the
locally $L^{0}-$convex topology for $L^{0}(\mathcal{F},K)$ is not necessarily a linear topology (see \cite{FKV} for details), and in particular when $(E,\mathcal{P})$ is endowed with the locally $L^{0}-$convex topology  $(E,\mathcal{P})$ is a topological module over the topological ring $(L^{0}(\mathcal{F},K),\mathcal
{T}_{c})$.

We can now introduce the following:

\begin{definition} \label{def:contmodule}Let $(E,\mathcal{P})$ be a
random locally convex module over $K$ with base
$(\Omega,\mathcal{F},P)$ and define $E_{\varepsilon,\lambda}^{\ast}$,
$E_{c}^{\ast}$, $E_{max}^{\ast}$ and $E_{min}^{\ast}$ as follows:
$$(1)~~~E_{\varepsilon,\lambda}^{\ast}=\{f~|~f~is~a~continuous~module~homomorphism~from~(E,{\cal
T}_{\varepsilon,\lambda})~to~(L^{0}(\mathcal {F},K),{\cal
T}_{\varepsilon,\lambda})\},~$$
$$(2)~~~E_{c}^{\ast}=\{f~|~f~is~a~continuous~module~homomorphism~from~(E,{\cal
T}_{c})~to~(L^{0}(\mathcal {F},K),{\cal T}_{c})\},~~~~~~$$
$$(3)~~~E_{max}^{\ast}=\{f~|~f~is~a~continuous~module~homomorphism~from~(E,{\cal
T}_{c})~to~(L^{0}(\mathcal {F},K),{\cal
T}_{\varepsilon,\lambda})\},$$ where $(L^{0}(\mathcal {F},K),{\cal
T}_{\varepsilon,\lambda})$ is regarded as a topological ring;
$$(4)~~~E_{min}^{\ast}=\{f~|~f~is~a~continuous~module~homomorphism~from~(E,{\cal
T}_{\varepsilon,\lambda})~to~(L^{0}(\mathcal {F},K),{\cal
T}_{c})\},~$$ where $(E,{\cal T}_{\varepsilon,\lambda})$ is regarded
as a topological module over the topological ring $(L^{0}(\mathcal
{F},K),{\cal T}_{\varepsilon,\lambda})$.
\end{definition}

$E_{\varepsilon,\lambda}^{\ast}$ was first introduced in \cite{Guo-Survey} and deeply developed in \cite{Guo-Peng,Guo-Zhu}. Companying the locally $L^{0}-$convex topology $E_{c}^{\ast}$ first occurred in \cite{FKV} in an anonymous way, but one of the recent results in \cite{Guo-Relations between} showed that there had been another equivalent formulation of $E_{c}^{\ast}$ before the locally $L^{0}-$convex topology occurred, namely $E^{\ast}_{I}$ (see Definition \ref{def:homomorphismI} below).   It is well known from \cite{Guo-Relations between} that
$E_{c}^{\ast}\subset E_{\varepsilon,\lambda}^{\ast}$ generally, and
$E_{c}^{\ast}= E_{\varepsilon,\lambda}^{\ast}$ if $\mathcal{P}$ has
the countable concatenation property (see Section 2).

With the above preparations, we present the main results of this
paper as follows:

\begin{theorem}\label{thm:maxel}Let $(E,\mathcal {P})$ be a random
locally convex module over $K$ with base $(\Omega,\mathcal {F}, P)$.
Then $E^{\ast}_{max}=E^{\ast}_{\varepsilon,\lambda}$.
\end{theorem}

\begin{theorem}\label{thm:minel}  Let $(E,\mathcal {P})$ be a random
locally convex module over $K$ with base $(\Omega,\mathcal {F}, P)$ and
$f\in E^{\ast}_{min}$. If $(\Omega,\mathcal {F}, P)$ is a nonatomic
probability space, then $f(x)=0, \forall x\in E$.
\end{theorem}

From the above two theorems, it is easy to see that among the four
kinds of random conjugate spaces only
$E^{\ast}_{\varepsilon,\lambda}$ and $E^{\ast}_{c}$ are universally suitable for the current development of the theory of random
conjugate spaces.

The remainder of this paper is organized as follows: in Section 2 we will
briefly collect some necessary known facts and in Section 3 we will prove our main
results.

\section{Preliminaries}

Denote by $\bar{L}^{0}(\mathcal{F},R)$ the set of all equivalence
classes of extended $R-$valued measurable functions on
$(\Omega,\mathcal {F},P)$. Then it is well known from
\cite{Dunford-Schwartz} that $\bar{L}^{0}(\mathcal {F},R)$ is a
complete lattice under the ordering $\leq$: $\xi\leq\eta$ iff
$\xi^{0}(\omega)\leq\eta^{0}(\omega)$, for almost all $\omega$ in
$\Omega$ (briefly, a.s.), where $\xi^{0}$ and $\eta^{0}$ are
arbitrarily chosen representatives of $\xi$ and $\eta$,
respectively. Furthermore, every subset $G$ of $\bar{L}^{0}(\mathcal
{F},R)$ has a supremum, denoted by $\bigvee G$, and an infimum,
denoted by $\bigwedge G$. Finally $L^{0}(\mathcal {F},R)$, as a
sublattice of $\bar{L}^{0}(\mathcal {F},R)$, is also a complete
lattice in the sense that every subset with upper bound has a
supremum. The pleasant properties of $\bar{L}^{0}(\mathcal{F},R)$
are summarized as follows:

\begin{proposition}[\cite{Dunford-Schwartz}]\label{prop:directe} For
every subset $G$ of $\bar{L}^{0}(\mathcal {F},R)$ there exist
countable subsets $\{a_{n}~|~n\in N\}$ and $\{b_{n}~|~n\in N\}$ of
$G$ such that $\bigvee G=\bigvee_{n\geq1}a_{n}$ and $\bigwedge
G=\bigwedge_{n\geq1}b_{n}$. Further, if $G$ is directed $($dually
directed$)$ with respect to $\leq$, then the above $\{a_{n}~|~n\in
N\}$ $($accordingly, $\{b_{n}~|~n\in N\}$$)$ can be chosen as
nondecreasing $($correspondingly, nonincreasing$)$ with respect to
$\leq$.
\end{proposition}

Specially, $L^{0}_{+}=\{\xi\in L^{0}(\mathcal{F},R)~|~\xi\geq 0\}$,
$L^{0}_{++}=\{\xi\in L^{0}(\mathcal{F},R)~|~\xi>0$ on $\Omega\}$,
where for $A\in \mathcal{F}$, $``\xi>\eta"$ on $A$ means
$\xi^{0}(\omega)>\eta^{0}(\omega)$ a.s. on $A$ for any chosen
representatives $\xi^{0}$ and $\eta^{0}$ of $\xi$ and $\eta$,
respectively. As usual, $\xi>\eta$ means $\xi\geq\eta$ and
$\xi\neq\eta$. For an arbitrarily chosen representative $\xi^{0}$ of
$\xi\in L^{0}(\mathcal{F},K)$, define the two
$\mathcal{F}-$measurable random variables $(\xi^{0})^{-1}$ and
$|\xi^{0}|$ by $(\xi^{0})^{-1}(\omega)=1/\xi^{0}(\omega)$ if
$\xi^{0}(\omega)\neq 0$, and $(\xi^{0})^{-1}(\omega)=0$ otherwise,
and by $|\xi^{0}|(\omega)=|\xi^{0}(\omega)|$, $\forall\omega\in
\Omega$. Then the equivalent class $Q(\xi)$ of $(\xi^{0})^{-1}$ is
called the generalized inverse of $\xi$; the equivalent class
$|\xi|$ of $|\xi^{0}|$ is called the absolute value of $\xi$.

For any $A\in \mathcal{F}$, $A^{c}$ denotes the complement of
$A$, $\tilde{A}=\{B\in \mathcal{F}~|~P(A\Delta B)=0\}$
denotes the equivalence class of $A$, where $\Delta$ is the
symmetric difference operation, $I_{A}$ the characteristic function
of $A$, and $\tilde{I}_{A}$ is used to denote the equivalence class
of $I_{A}$; given two $\xi$ and $\eta$ in $L^{0}(\mathcal{F}, R)$,
and $A=\{\omega\in\Omega~|~\xi^{0}\neq\eta^{0}\}$, where $\xi^{0}$
and $\eta^{0}$ are arbitrarily chosen representatives of $\xi$ and
$\eta$ respectively, then we always write $[\xi\neq\eta]$ for the
equivalence class of $A$ and $I_{[\xi\neq\eta]}$ for
$\tilde{I}_{A}$, one can also understand the implication of such
notations as $I_{[\xi\leq\eta]}$, $I_{[\xi<\eta]}$ and
$I_{[\xi=\eta]}$.

\begin{definition}[\cite{Guo-Survey, G-C}]\label{def:linear}
$(1)$   Let $E$ be a linear space over $K$, then a mapping $f:
E\rightarrow L^{0}(\mathcal {F}, K)$ is called a random linear
functional on $E$ if $f$ is linear;

$(2)$  If $E$ is a linear space over
$R$, then a mapping $f: E\rightarrow L^{0}(\mathcal {F},R)$ is
called a random sublinear functional on $E$ if $f(\alpha
x)=\alpha\cdot f(x)$ for any positive real number $\alpha$
and $x\in E$, and $f(x+y)\leq f(x)+f(y),\forall x,y\in E$;

$(3)$  Let $E$
be a linear space over $K$, then a mapping $f: E\rightarrow
L_{+}^{0}$ is called a random seminorm on $E$ if $f(\alpha
x)=|\alpha|\cdot f(x), \forall \alpha \in K$ and $x\in E$, and
$f(x+y)\leq f(x)+f(y), \forall x,y\in E$;

$(4)$  Let $E$ be
a left module over the algebra $L^{0}(\mathcal {F},K)$, then a
mapping $f: E\rightarrow L^{0}(\mathcal {F}, K)$ is called a
$L^{0}-$linear functional on $E$ if $f$ is a module homomorphism;

$(5)$  Let $E$ be a left module over the algebra $L^{0}(\mathcal {F},R)$, a
mapping $f: E\rightarrow L^{0}(\mathcal {F}, R)$ is called an
$L^{0}$-sublinear functional on $E$ if $f$ is a random sublinear
functional on $E$ such that $f(\xi\cdot x)=\xi\cdot f(x), \forall\xi \in
L_{+}^{0}$ and $x\in E$;

$(6)$  Let $E$ be a left module over the algebra
$L^{0}(\mathcal {F},K)$, then a mapping $f: E\rightarrow L_{+}^{0}$
is called an $L^{0}$-seminorm on $E$ if $f$ is a random seminorm on
$E$ such that $f(\xi\cdot x)=|\xi|\cdot f(x), \forall\xi \in
L^{0}(\mathcal {F},K)$ and $x\in E$.
\end{definition}

\begin{definition}[\cite{Guo-Relations
between,Guo-Survey}]\label{def:locally module} An ordered pair
$(E,\mathcal {P})$ is called a random locally convex space over $K$
with base $(\Omega , \mathcal{F}, P)$ if the following three
conditions are satisfied:

(1)  $E$ is a linear space over $K$;

(2)  $\mathcal {P}$ is a family of random seminorms on $E$ with
base $(\Omega,\mathcal{F},P)$;

(3)  $\bigvee\{\|x\|~|~\|\cdot\|\in{\cal P}\}=0$ implies
$x=\theta$ $($the null element of $E$ $)$.\\
In addition, if $E$ is a left module over the algebra
$L^{0}(\mathcal{F},K)$ and each $\|\cdot\|$ in $\mathcal{P}$ is an
$L^{0}-$seminorm then such a random locally convex space is called
a random locally convex module over $K$ with base $(\Omega ,
\mathcal{F}, P)$.
\end{definition}

\begin{remark} Let $(E,\mathcal {P})$ be a random
locally convex space (a random locally convex module) over $K$ with
base $(\Omega , \mathcal{F}, P)$. If $\mathcal{P}$ degenerates to a
singleton $\{\|\cdot\|\}$, then $(E,\|\cdot\|)$ is exactly a random
normed space, briefly, an $RN$ space (correspondingly, a random
normed module, briefly, an $RN$ module). Specially,
$(L^{0}(\mathcal{F},K),|\cdot|)$ is an $RN$ module.
\end{remark}

In the sequel, for a random locally convex space $(E,{\cal P})$ with
base $(\Omega,{\cal F},P)$ and for each finite subfamily
$\mathcal{Q}$ of ${\cal P}$, $\|\cdot\|_{\mathcal{Q}}:E\rightarrow
L^{0}_{+}({\cal F})$ always denotes the random seminorm of $E$
defined by $\|x\|_{\mathcal{Q}}=\bigvee\{\|x\|~|~\|\cdot\|\in
\mathcal{Q}\},\forall x\in E$, and ${\cal F}({\cal P})$ the set of
finite subfamilies of ${\cal P}$.

For each random locally convex space $(E,\mathcal {P})$ over $K$
with base $(\Omega , \mathcal{F}, P)$, $\mathcal {P}$ can induce two kinds of
topologies, namely the
$(\varepsilon,\lambda)-$topology and the locally $L^{0}-$convex
topology.

\begin{definition}[\cite{Guo-Relations
between,Guo-progress,Guo-Survey}]\label{def:mintopology} Let $(E,\mathcal{P})$ be a
random locally convex space over $K$ with base
$(\Omega,\mathcal{F},P)$. For any positive real numbers
$\varepsilon$ and $\lambda$ such that $0<\lambda<1$, and any
$\mathcal {Q}\in \mathcal{F}(\mathcal{P})$, let $N_{\theta}(\mathcal
{Q},\varepsilon,\lambda)=\{x\in
E~|~P\{\omega\in\Omega~|~\|x\|_{\mathcal
{Q}}(\omega)<\varepsilon\}>1-\lambda\}$, then $\{N_{\theta}(\mathcal
{Q},\varepsilon,\lambda)~|~\mathcal {Q}\in
\mathcal{F}(\mathcal{P}),\varepsilon>0,0<\lambda<1\}$ is easily
verified to be a local base at the null vector $\theta$ of some
Hausdorff linear topology. The linear topology is called the
$(\varepsilon,\lambda)-$topology for $E$ induced by $\mathcal{P}$.
\end{definition}

From now on, the $(\varepsilon,\lambda)-$topology for each random
locally convex space is always denoted by $\mathcal
{T}_{\varepsilon,\lambda}$ when no confusion occurs.

\begin{proposition}[\cite{Guo-Relations
between,Guo-progress,Guo-Survey}] Let $(E,{\cal P})$ be a
random locally convex space over $K$ with base $(\Omega,{\cal
F},P)$.
Then we have the following statements:\\
\indent$(1)$ The $(\epsilon,\lambda)-$topology for $L^{0}({\cal
F},K)$ is exactly the topology of convergence in probability $P$,
and
$(L^{0}({\cal F},K),{\cal T}_{\epsilon,\lambda})$ is a topological algebra over $K$;\\
\indent$(2)$ If $(E,{\cal P})$ is a random locally convex module,
then $(E,{\cal T}_{\epsilon,\lambda})$ is a topological module over the topological algebra $L^{0}({\cal F},K)$;\\
\indent$(3)$ A net $\{x_{\delta},\delta\in\Gamma\}$ converges in the
$(\epsilon,\lambda)-$topology to some $x$ in $E$ iff for each
$\|\cdot\|\in{\cal P}$ $\{\|x_{\delta}-x\|,\delta\in\Gamma\}$
converges in probability $P$ to $0$.
\end{proposition}

The following locally $L^{0}-$convex topology is easily seen to be
much stronger than the $(\varepsilon,\lambda)-$topology, and was
first introduced by Filipovi$\acute{c}$, Kupper and Vogelpoth in
\cite{FKV} for random locally convex modules.

\begin{definition}[\cite{FKV}]\label{def:maxtopology} Let $(E,\mathcal{P})$ be a random locally
convex space over $K$ with base $(\Omega,\mathcal{F},P)$. For any
$\mathcal {Q}\in \mathcal{F}(\mathcal{P})$ and $\varepsilon\in
L^{0}_{++}$, let $N_{\theta}(\mathcal {Q},\varepsilon)=\{x\in
E~|~\|x\|_{\mathcal{Q}}\leq \varepsilon\}$. A subset $G$ of $E$ is
called $\mathcal{T}_{c}-$open if for each $x\in G$ there exists some
$N_{\theta}(\mathcal {Q},\varepsilon)$ such that
$x+N_{\theta}(\mathcal {Q},\varepsilon)\subset G$, $\mathcal
{T}_{c}$ denotes the family of $\mathcal{T}_{c}-$open subsets of
$E$. Then it is easy to see that $(E,\mathcal {T}_{c})$ is a
Hausdorff topological group with respect to the addition on $E$.
$\mathcal{T}_{c}$ is called the locally $L^{0}-$convex topology for $E$
induced by $\mathcal{P}$.
\end{definition}

From now on, the locally $L^{0}-$convex topology for each random
locally convex space is always denoted by $\mathcal{T}_{c}$ when no
confusion occurs.

\begin{proposition}[\cite{FKV}]
Let $(E,\cal P)$ be a random locally convex module over $K$ with
base $(\Omega,{\cal F},P)$. Then

$(1)$ $L^{0}({\cal F},K)$ is a topological ring endowed with its
locally $L^{0}-$convex topology;

$(2)$ $E$ is a topological module over the topological ring
$L^{0}({\cal F},K)$ when $E$ and $L^{0}({\cal F},K)$ are endowed
with their respective locally $L^{0}-$convex topologies;

$(3)$ A net $\{x_{\alpha}~|~ \alpha\in\Gamma\}$ in $E$ converges in
the locally $L^{0}-$convex topology to $x\in E$ iff
$\{\|x_{\alpha}-x\|~|~ \alpha\in\Gamma\}$ converges in the locally
$L^{0}-$convex topology of $L^{0}({\cal F},K)$ to $0$ for each
$\|\cdot\|\in\cal P$.
\end{proposition}

${\cal T}_{c}$ is called locally $L^{0}-$convex because it has a
striking local base ${\cal
U}_{\theta}=\{B_{\mathcal{Q}}(\varepsilon)~|~\mathcal{Q}\subset
{\cal P}$ finite and $\varepsilon\in L^{0}_{++}\}$, each member $U$
of which is:

\hspace{2mm}(\lowercase \expandafter {\romannumeral 1})
$L^{0}-$convex: $\xi\cdot x+(1-\xi)\cdot y\in U$ for any $x,y\in U$
and $\xi\in L^{0}_+$ such that $0\leqslant\xi\leqslant1$;

\hspace{1mm}(\lowercase \expandafter {\romannumeral 2})
$L^{0}-$absorbent: there is $\xi\in L^{0}_{++}$ for each $x\in E$
such that $x\in\xi\cdot U$;

(\lowercase \expandafter {\romannumeral 3}) $L^{0}-$balanced:
$\xi\cdot x\in U$ for any $x\in U$ and any $\xi\in L^{0}({\cal
F},K)$ such that $|\xi|\leqslant1$.

\begin{remark} Let $(E,\mathcal {P})$ be a random
locally convex module over $K$ with base $(\Omega,\mathcal {F}, P)$
endowed with the locally $L^{0}-$convex topology $\mathcal {T}_{c}$.
Although $E$ is a linear space, $(E,{\cal T}_c)$
may not be a topological linear space since the scalar
multiplication is not necessarily
continuous, see \cite{FKV} for details.
\end{remark}

Historically, the earliest two notions of a random conjugate space
of a random locally convex space were introduced in
\cite{Guo-Survey} and \cite{Guo-Module}, respectively. As shown
in \cite{Guo-Relations between, Guo-progress}, it turned out
that they just correspond to the $(\varepsilon,\lambda)-$topology
and the locally $L^{0}-$convex topology in the context of a random
locally convex module, respectively!

\begin{definition}[\cite{Guo-Module}]\label{def:homomorphismI} Let
$(E,\mathcal{P})$ be a random locally convex space over $K$ with
base $(\Omega,\mathcal{F},P)$. A random linear functional
$f:E\rightarrow L^{0}(\mathcal{F},K)$ is called an a.s. bounded
random linear functional of type I if there are some $\xi\in
L^{0}_{+}$ and $\mathcal{Q}\in \mathcal{F}(\mathcal{P})$ such that
$|f(x)|\leq\xi\cdot\|x\|_{\mathcal{Q}},\forall x\in E$. Denote by
$E^{\ast}_{I}$ the set of a.s. bounded random linear functional of
type I on $E$. The module multiplication operation
$\cdot:L^{0}(\mathcal{F},K)\times E^{\ast}_{I}\rightarrow
E^{\ast}_{I}$ is defined by $(\xi f)(x)=\xi(f(x)),\forall \xi\in
L^{0}(\mathcal{F},K),f\in E^{\ast}_{I}$ and $x\in E$. It is easy to
see that $E^{\ast}_{I}$ is a left module over
$L^{0}(\mathcal{F},K)$, called the random conjugate space of type I
of $E$.
\end{definition}

\begin{definition}[\cite{Guo-Survey,G-C}]\label{def:homomorphismII} Let
$(E,\mathcal{P})$ be a random locally convex space over $K$ with
base $(\Omega,\mathcal{F},P)$. A random linear functional
$f:E\rightarrow L^{0}(\mathcal{F},K)$ is called an a.s. bounded
random linear functional of type II on $E$ if there exist a
countable partition $\{A_{i}~|~i\in N\}$ of $\Omega$ to
$\mathcal{F}$, a sequence $\{\xi_{i}~|~i\in N\}$ in $L^{0}_{+}$ and
a sequence $\{\mathcal {Q}_{i}~|~i\in N\}$ in
$\mathcal{F}(\mathcal{P})$ such that
$|f(x)|\leq\Sigma_{i=1}^{\infty}\tilde{I}_{A_{i}}\cdot\xi_{i}\cdot
\|x\|_{\mathcal{Q}_{i}},\forall x\in E$. Denote by
$E^{\ast}_{II}$ the $L^{0}(\mathcal{F},K)-$module
of a.s. bounded random linear functional of type II on $E$, called
the random conjugate space of type II of $E$.
\end{definition}

Propositions 2.12 and 2.13 below give the
topological characterizations of an element in $E^{\ast}_{I}$ and $E^{\ast}_{II}$, respectively.

\begin{proposition}[\cite{Guo-Relations
between,Guo-progress,Guo-Survey}]\label{prop:continuousII} Let $(E,{\cal P})$ be a
random locally convex module over $K$ with base $(\Omega,{\cal
F},P)$ and $f:E\rightarrow L^{0}({\cal F},K)$ a random linear
functional. Then $f\in E^{\ast}_I$ iff $f$ is a continuous module
homomorphism from $(E,{\cal T}_c)$ to $(L^{0}({\cal F},K),{\cal
T}_c)$, namely $E^{\ast}_I=E^{\ast}_c$.
\end{proposition}

\begin{proposition}[\cite{Guo-Relations
between, Guo-Zhu}]\label{prop:continuousI} Let $(E,{\cal P})$ be a random
locally convex module over $K$ with base $(\Omega,{\cal F},P)$ and
$f:E\rightarrow L^{0}({\cal F},K)$ a random linear functional. Then
$f\in E^{\ast}_{II}$ iff $f$ is a continuous module homomorphism
from $(E,{\cal T}_{\epsilon,\lambda})$ to $(L^{0}({\cal F},K),{\cal
T}_{\epsilon,\lambda})$, namely
$E^{\ast}_{II}=E^{\ast}_{\varepsilon,\lambda}$.
\end{proposition}

\begin{definition}[\cite{Guo-Relations
between, FKV}] Let $(E,{\cal P})$ be a random
locally convex module over $K$ with base $(\Omega,{\cal F},P)$.
${\cal P}$ is called having the countable concatenation property if
$\sum_{n=1}^{\infty}\tilde{I}_{A_n} \|\cdot\|_{\mathcal {Q}_n}$
still belongs to ${\cal P}$ for any countable partition
$\{A_n\,|\,n\in N\}$ of $\Omega$ to ${\cal F}$ and any sequence
$\{\mathcal {Q}_{n}\,|\,n\in N\}$ in $\mathcal{F}(\mathcal{P})$.
\end{definition}

\begin{proposition}[\cite{Guo-Relations between}] Let
$(E,\mathcal{P})$ be a random locally convex module. Then
$E^{\ast}_{\varepsilon,\lambda}=E^{\ast}_{c}$ if $\mathcal{P}$ has
the countable concatenation property $($generally, it is obvious
that $E^{\ast}_{c}\subset E^{\ast}_{\varepsilon,\lambda}$$)$. In
particular $E^{\ast}_{\varepsilon,\lambda}=E^{\ast}_{c}$ for any
$RN$ module $(E,\|\cdot\|)$.
\end{proposition}

\section{Proofs of the Main Results}

\begin{lemma}[\cite{Guo-Zhu}]\label{lem:[8]} Suppose $E$ is a
left module over the algebra $L^{0}(\mathcal{F},K)$, $f:E\rightarrow
L^{0}(\mathcal{F},K)$ and $\|\cdot\|:E\rightarrow L^{0}_{+}$ are
such that $|f(\xi\cdot x)|=\xi\cdot|f(x)|$ and $\|\xi\cdot
x\|=\xi\cdot\|x\|, \forall \xi\in L^{0}_{+}$ and $x\in E$. Denote
$B_{\|\cdot\|}(1)=\{x\in E~|~\|x\|\leq 1\}$, then there exists
$\eta\in L^{0}_{+}$ such that $|f(x)|\leq \eta\cdot\|x\|,\forall
x\in E$ iff $\bigvee\{|f(x)|~|~x\in B_{\|\cdot\|}(1)\}\in
L^{0}_{+}$.
\end{lemma}

We can now prove Theorem \ref{thm:maxel}.

\noindent {\bf Proof of Theorem \ref{thm:maxel}.}\quad  Since
$(L^{0}(\mathcal{F},K),|\cdot|)$ is endowed with the
$(\varepsilon,\lambda)-$topology,
$\{V_{0}(\varepsilon,\lambda)~|~\varepsilon$ and $\lambda$ are real
numbers, $\varepsilon>0$ and $0<\lambda< 1\}$ is a local base at $0$
of the $(\varepsilon,\lambda)-$topology for $L^{0}(\mathcal{F},K)$,
where $V_{0}(\varepsilon,\lambda)=\{\xi\in
L^{0}(\mathcal{F},K)~|~P(\{\omega\in\Omega~|~|\xi|(\omega)<\varepsilon\})>1-\lambda\}$.
Select $\{\varepsilon_{n}~|~n\in N\}$ and $\{\lambda_{n}~|~n\in N\}$
to be any two sequences of positive numbers which both tend to 0 in
a decreasing way and $\lambda_{n}<1$ for each $n\in N$, then clearly
$\{V_{0}(\varepsilon_{n},\lambda_{n})~|~n\in N\}$ is also a local
base at $0$ of the $(\varepsilon,\lambda)-$topology for
$L^{0}(\mathcal{F},K)$.

Since $f:(E,{\cal T}_c)\rightarrow (L^{0}(\mathcal{F},K),{\cal T}_{\epsilon,\lambda})$ is continuous, for each
$n\in N$ there exist at least some $\bar{\varepsilon}_{n}\in
L^{0}_{++}$ and $\|\cdot\|_{n}\in \mathcal{F}(\mathcal{P})$ such
that
$$f(B_{n}(\bar{\varepsilon}_{n}))\subset
V_{0}(\varepsilon_{n},\lambda_{n}),$$ where
$B_{n}(\bar{\varepsilon}_{n})=\{x\in E~|~\|x\|_{n}\leq
\bar{\varepsilon}_{n}\}$. Let
$\|\cdot\|_{\mathcal{Q}_{1}}=\|\cdot\|_{1}$,
$\|\cdot\|_{\mathcal{Q}_{n}}=\|\cdot\|_{1}\bigvee\cdots
\bigvee\|\cdot\|_{n}, \forall n\geq 2$ and
$B_{\mathcal{Q}_{n}}(\bar{\varepsilon}_{n})=\{x\in
E~|~\|x\|_{\mathcal{Q}_{n}}\leq\bar{\varepsilon}_{n}\}$, then
$$f(B_{\mathcal{Q}_{n}}(\bar{\varepsilon}_{n}))\subset
V_{0}(\varepsilon_{n},\lambda_{n})$$ and
$$B_{\mathcal{Q}_{n+1}}(1)\subset B_{\mathcal{Q}_{n}}(1), \forall n\in
N,$$ where $B_{\mathcal{Q}_{n}}(1)=\{x\in
E~|~\|x\|_{\mathcal{Q}_{n}}\leq 1\},~\forall n\in N$.

Denote $\{|f(x)|~|~x\in
B_{\mathcal{Q}_{n}}(\bar{\varepsilon}_{n})\}$ by $G_{n}$ and
$\bigvee G_{n}$ by $\eta_{n}$. First, it is easy to see that $G_{n}$
is a directed set in $L^{0}_{+}$, then there exists a sequence
$\{x_{n,k}~|~k\in N\}$ in
$B_{\mathcal{Q}_{n}}(\bar{\varepsilon}_{n})$ such that
$$\{|f(x_{n,k})|~|~k\in N\}\nearrow \eta_{n}.$$ Since $P(\{\omega\in \Omega~|~|f(x_{n,k})|(\omega)
<\varepsilon_{n} \})>1-\lambda_{n}, \forall k\in N$, then
$$P(\{\omega\in \Omega~|~\eta_{n}(\omega)\leq 2\varepsilon_{n}
\})>1-\lambda_{n}, \forall n\in N.$$ From $\eta_{n}=\bigvee
G_{n}=\bigvee\{|f(x)|~|~x\in
B_{\mathcal{Q}_{n}}(\bar{\varepsilon}_{n})\}$, we can obtain that
$$\frac{1}{\bar{\varepsilon}_{n}}\cdot\eta_{n}=\bigvee\{|f(x)|~|~x\in
B_{\mathcal{Q}_{n}}(1)\}, \forall n\in N$$ and $$P(\{\omega\in
\Omega~|~\frac{1}{\bar{\varepsilon}_{n}(\omega)}\cdot\eta_{n}(\omega)<+\infty\})=P(\{\omega\in
\Omega~|~\eta_{n}(\omega)<+\infty\})$$ $$\geq
P(\{\omega\in\Omega~|~\eta_{n}(\omega)\leq
2\varepsilon_{n}\})\geq1-\lambda_{n}.$$ Since
$B_{\mathcal{Q}_{n+1}}(1)\subset B_{\mathcal{Q}_{n}}(1)$ and
$\lambda_{n}\searrow 0$, it is clear that $$P(\{\omega\in
\Omega~|~\eta_{n}(\omega)<+\infty\})\leq P(\{\omega\in
\Omega~|~\eta_{n+1}(\omega)<+\infty\}),$$ $$\{P(\{\omega\in
\Omega~|~\eta_{n}(\omega)<+\infty\})~|~n\in N\}\nearrow 1$$ and
$$\tilde{\Omega}=\cup_{n=1}^{\infty}[\eta_{n}<\infty].$$ Taking $A_{1}=\{\omega\in \Omega~|~\frac{1}
{\bar{\varepsilon}_{1}^{0}(\omega)}\cdot\eta^{0}_{1}(\omega)<+\infty\}$,
$A_{i}=\{\omega\in \Omega~|~\frac{1}
{\bar{\varepsilon}_{i}^{0}(\omega)}\cdot\eta^{0}_{i}(\omega)<\frac{1}
{\bar{\varepsilon}_{i-1}^{0}(\omega)}\cdot\eta^{0}_{i-1}(\omega)=+\infty\}$,
$\forall i\geq 2$, then $\{A_{n}~|~n\in N\}$ forms a countable
partition of $\Omega$ to $\mathcal{F}$. For each $n\in N$, define
$P_{n}:A_{n}\cap\mathcal{F}\rightarrow [0,1]$ by $P_{n}(A_n\cap F)=\frac{P(A_n\cap F)}{P(A_n)},\forall F\in{\cal F}$,
$E^{(n)}=\tilde{I}_{A_{n}}\cdot E$, $\|\cdot\|^{(n)}=$ the
limitation of $\|\cdot\|_{\mathcal{Q}_{n}}$ to $E^{(n)}$, and
$f^{(n)}=$ the limitation of $f$ to $E^{(n)}$, then $E^{(n)}$ is a
left module over the algebra $L^{0}(A_{n}\cap\mathcal{F},K)$.
Applying Lemma \ref{lem:[8]} to $E^{(n)}$, $f^{(n)}$ and
$\|\cdot\|_{\mathcal{Q}_{n}}$yields
$$\tilde{I}_{A_{n}}|f(x)|\leq \frac{\eta_{n}}{\bar{\varepsilon}_{n}}\cdot \tilde{I}_{A_{n}}\cdot
\|x\|_{\mathcal{Q}_{n}},\forall n\in N.$$
Let $\xi_{n}=\frac{\eta_{n}}{\bar{\varepsilon}_{n}}\cdot
\tilde{I}_{A_{n}}$, then
$$|f(x)|\leq\Sigma_{n=1}^{\infty}\tilde{I}_{A_{n}}\cdot\xi_{n}\cdot\|x\|_{\mathcal{Q}_{n}},$$
$\forall x\in E$. By Definition \ref{def:homomorphismI} it is obvious that
$E^{\ast}_{max}\subset E^{\ast}_{\varepsilon,\lambda}$. Finally, since it
is clear that $E^{\ast}_{\varepsilon,\lambda}\subset
E^{\ast}_{max}$, then we can obtain that $E^{\ast}_{max}=
E^{\ast}_{\varepsilon,\lambda}$.

This completes the proof of Theorem \ref{thm:maxel}.

Before giving the proof of Theorem \ref{thm:minel}, we need the following topological terminology and several lemmas on totally disconnected spaces.

\begin{definition}[\cite{Conway}]\label{def:convey} A topological
space $(X,\mathcal{T})$ is called totally disconnected if for any $
x\in X$ and any neighborhood $U$ of $x$, there is
$V\subset X$ that is both $\mathcal{T}-$ open and
$\mathcal{T}-$closed such that $$x\in V\subset U.$$
\end{definition}

From Definition \ref{def:convey}, we can easily obtain the following:

\begin{lemma}\label{lem:3.3} Let $(X,\mathcal{T})$ be a Hausdorff
and totally disconnected space. Then any nonempty $A\subset X$ is a connected
subset iff $A$ is a single point set.
\end{lemma}

Now let us recall the notion of an atom: Let $(\Omega,\mathcal
{F},P)$ be a probability space, a set $A\in \mathcal {F}$ is called
an atom (or, $P-$atom) if $P(A)>0$ and if $B\in \mathcal {F}$,
$B\subset A$, then either $P(B)=0$ or $P(A\setminus B)=0$. It is
clear that if $A_{1}$ and $A_{2}$ are atoms, then either
$P(A_{1}\cap A_{2})=0$ or $P(A_{1}\Delta A_{2})=0$. Also, it is easy to see that
$(\Omega,\mathcal {F},P)$ essentially has at most countably many disjoint
atoms. In this paper, we say that $(\Omega,\mathcal {F},P)$ is
essentially purely $P-$atomic if there exists at most, a countable
family $\{A_{n}~|~n\in N\}$ of disjoint atoms such that
$\Omega=\sum_{n=1}^{\infty}A_{n}$ and such that for each $A\in
\mathcal {F}$ there is $B$ in the $\sigma-$algebra generated by the
family $\{A_{n}~|~n\in N\}$ such that $P(A\Delta B)=0$. A
probability space $(\Omega,\mathcal {F},P)$ without any atoms is
called nonatomic.

For a nonatomic probability space $(\Omega,\mathcal {F},P)$, the
following two lemmas are known and very important for the
proof of Lemma \ref{lem:3.6} below.

\begin{lemma}[\cite{Dudley}]\label{lem:Dudley} If $(\Omega,\mathcal
{F},P)$ is a nonatomic probability space, then the range of $P$ is
the whole interval $[0,1]$.
\end{lemma}

\begin{lemma}\label{lem:3.5} Let $(\Omega,\mathcal {F},P)$ be a
nonatomic probability space and $A\in\mathcal {F}$ with $P(A)>0$.
Then there is a countable partition $\{A_n\mid n\in N\}$ of $A$ to
$\mathcal {F}$ such that $P(A_{n})=\frac{1}{2^{n}}P(A)$.
\end{lemma}

\begin{lemma}\label{lem:3.6} Let $(\Omega,\mathcal{F},P)$ be a
nonatomic probability space, $(E,\|\cdot\|)$ an $RN$
module over $K$ with the base $(\Omega,\mathcal{F},P)$,
$B^{\circ}(1)=\{x\in E~|~\|x\|<1$ on $\Omega\}$ and $M=\{x\in
B^{\circ}(1)~|~\exists m_{x}\in R, 0<m_{x}<1$ such that $\|x\|<m_{x}$
on $\Omega\}$. Then $M$ is an $L^{0}-$convex, $\mathcal
{T}_{c}-$closed and $\mathcal {T}_{c}-$open subset of $E$.
\end{lemma}

\begin{proof} If $x_{1}$, $x_{2}\in M$, then according to the
definition of $M$ there are two positive real numbers $m_{x_{1}}<1$
and $m_{x_{2}}<1$ such that $\|x_{1}\|<m_{x_{1}}$ and
$\|x_{2}\|<m_{x_{2}}$ on $\Omega$. It is easy to see that
$\|\xi\cdot x_{1}+(1-\xi)\cdot
x_{2}\|\leq|\xi|\cdot\|x_{1}\|+|1-\xi|\cdot\|x_{2}\|<max(m_{x_{1}},m_{x_{2}})<
1$ on $\Omega$, where $\xi\in L^{0}_{+}$ with $0 \leq\xi\leq 1$.
Thus, $M$ is $L^{0}-$convex.

Now, we prove that $M$ is a $\mathcal {T}_{c}-$closed subset of $E$.
We only need to check that $E\setminus M$ is $\mathcal
{T}_{c}-$open and this can proceed in the following two cases:

\vskip0.1in

Case(1): when $y_{1}\in E\setminus M$ and $y_{1}\not\in
B^{\circ}(1)$, then there is $D_{1}\in \mathcal {F}$ with
$P(D_{1})>0$ such that $\|y_{1}\|\geq 1$ on $D_{1}$. By Lemma \ref{lem:3.5},
there is a countable partition $\{D_{1,n}~|~ n\in N\}$ of $D_{1}$
to $\mathcal {F}$ such that $P(D_{1,n})=\frac{1}{2^{n}}P(D_{1})$. Let
$$\varepsilon_{1}=\tilde{I}_{D_{1}^{c}}+\Sigma_{n=1}^{\infty}\frac{1}{n}\cdot\tilde{I}_{D_{1,n}}$$
and $$B(\varepsilon_{1})=\{x\in E~|~\|x\|\leq \varepsilon_{1}\},$$
then $B(\varepsilon_{1})$ is a neighborhood of 0, and for any
$\tilde{y}_{1}\in y_{1}+B(\varepsilon_{1})$ it is easy to see that
$$\|y_{1}\|-\|\tilde{y}_{1}\|\leq \|y_{1}-\tilde{y}_{1}\|\leq \varepsilon_{1}$$
and
$$\|\tilde{y}_{1}\|\geq\|y_{1}\|-\varepsilon_{1}\geq 1-\frac{1}{n}~on~D_{1,n},$$ namely
$P([\|\tilde{y}_{1}\|\geq1-\frac{1}{n}])\geq P(D_{1,n})>0$, $\forall
n\in N$. Consequently, we have that $\tilde{y}_{1}\not\in M$
and $y_{1}+B(\varepsilon_{1})\subset E\setminus M$.

Case(2): when $y_{2}\in E\setminus M$ and $y_{2}\in B^{\circ}(1)$,
then $\|y_{2}\|<1$ on $\Omega$ and $P(\{\omega\in \Omega
~|~\|y_{2}\|(\omega)>1-\frac{1}{n}\})>0$ for each $n\in N$ by the definition of $M$. Let $H_{n}=\{\omega\in
\Omega~|~\|y_{2}\|^{0}(\omega)>1-\frac{1}{n}\}$, $\forall n\in N$,
$D_{2,n}=H_{n}\setminus H_{n+1}=\{\omega\in
\Omega~|~1-\frac{1}{n}<\|y_{2}\|^{0}(\omega)\leq1-\frac{1}{n+1}\}$,
where $y_{2}^{0}$ is an arbitrarily chosen representative of
$y_{2}$, then for any $i,j\in N$ and $i\neq j$, $D_{2,i}\cap
D_{2,j}=\emptyset$ and $H_{i}=\sum_{n=i}^{\infty}D_{2,n}$. Assume
that there is $k\in N$ such that $P(D_{2,n})=0$, $\forall n\geq k$,
then it implies that $P(H_{k})=\Sigma_{n=k}^{\infty}P(D_{2,n})=0$, which is impossible. Hence, there exists a subsequence
$\{n_{k}~|~k\in N\}$ of $N$ such that $P(D_{2,n_{k}})>0$, $\forall
k\in N$ and we can suppose, without loss of generality,
$P(D_{2,n})>0$ for each $n\in N$. Let
$$D_{2}=\Sigma_{n=1}^{\infty}D_{2,n},$$
$$\varepsilon_{2}=\tilde{I}_{D_{2}^{c}}+\Sigma_{n=1}^{\infty}\frac{1}{n}\cdot
\tilde{I}_{D_{2,n}}$$ and
$$B(\varepsilon_{2})=\{x\in E~|~\|x\|\leq \varepsilon_{2}\},$$ then
$B(\varepsilon_{2})$ is a neighborhood of 0, and for any
$\tilde{y}_{2}\in y_{2}+B(\varepsilon_{2})$ it is obvious that
$$\|y_{2}\|-\|\tilde{y}_{2}\|\leq \|y_{2}-\tilde{y}_{2}\|\leq
\varepsilon_{2}$$ and
$$\|\tilde{y}_{2}\|\geq \|y_{2}\|-\varepsilon_{2}\geq 1-\frac{1}{n}-\frac{1}{n}=1-\frac{2}{n}~on~D_{2,n},$$
namely $P([\|\tilde{y}_{2}\|\geq1-\frac{2}{n}])\geq P(D_{2,n})>0$,
$\forall n\in N$. Consequently, we have that
$\tilde{y}_{2}\not\in M$ and $y_{2}+B(\varepsilon_{2})\subset
E\setminus M$.

From the two cases above, it is clear that $M$ is a $\mathcal
{T}_{c}-$closed subset of $E$.

Finally, we prove that $M$ is a
$\mathcal {T}_{c}-$open subset of $E$ as follows:
if $y\in M$, then $y\in B^{\circ}(1)$ and there is a positive real number
$m_{y}<1$ such that $\|y\|< m_{y}$ on $\Omega$. Let
$\varepsilon^{0}:\Omega\rightarrow R$ be defined by
$\varepsilon^{0}(\omega)=\frac{1-m_{y}}{2}$ for all $\omega\in
\Omega$ and $\varepsilon\in L_{++}^{0}$ the equivalence class of
$\varepsilon^{0}$, then $B(\varepsilon)=\{x\in E~|~\|x\|\leq
\varepsilon\}$ is a neighborhood of 0, and for any
$\tilde{y}\in y+B(\varepsilon)$
$$\|\tilde{y}\|-\|y\|\leq \|y-\tilde{y}\|\leq
\varepsilon$$ and
$$\|\tilde{y}\|\leq m_{y}+\frac{1-m_{y}}{2}=\frac{1+m_{y}}{2}<1$$ on $\Omega$.
Hence, $y+B(\varepsilon)\subset M$ and $M$ is a $\mathcal {T}_{c}-$open subset of
$E$.

This completes the proof.

\end{proof}

\begin{lemma}\label{lem:3.7} Let $(\Omega,\mathcal{F},P)$ be a
nonatomic probability space and $(E,\|\cdot\|)$ a random normed
module over $K$ with the base $(\Omega,\mathcal{F},P)$. Then
$(E,\mathcal{T}_{c})$ is a Hausdorff and totally disconnected space.
\end{lemma}

\begin{proof} We only need to check that $(E,\mathcal {T}_{c})$ is
totally disconnected. Let $M$ be the same
one as in Lemma \ref{lem:3.6},
$\mathcal{U}_{\theta}=\{B_{\|\cdot\|}(\varepsilon)~|~\varepsilon\in
L^{0}_{++}\}$, where $B_{\|\cdot\|}(\varepsilon)=\{x\in
E~|~\|x\|\leq\varepsilon\}$. Then $M$ is both $\mathcal {T}_{c}-$closed
and $\mathcal {T}_{c}-$open subset of $E$ by Lemma \ref{lem:3.6}. For each $\varepsilon\in L^{0}_{++}$,
it is easy to see that $M$ and $\varepsilon\cdot M$ are homeomorphic
and
$$0\in\varepsilon\cdot M\subset B_{\|\cdot\|}(\varepsilon).$$
Hence, by Definition \ref{def:maxtopology} and  Definition \ref{def:convey}, we have that
$(E,\mathcal {T}_{c})$ is
totally disconnected.

This completes the proof.

\end{proof}

We can now prove Theorem \ref{thm:minel}.

\noindent {\bf Proof of Theorem \ref{thm:minel}.}\quad Since $f\in E^{\ast}_{min}$,
$(E,\mathcal {T}_{\varepsilon,\lambda})$ is connected and
$(L^{0}(\mathcal{F},K),\mathcal {T}_{c})$ is a Haustorff and totally
disconnected space by Lemma \ref{lem:3.7}, it is clear that $f(E)$ is a
single point set of $L^{0}(\mathcal{F},K)$ by Lemma \ref{lem:3.3}. Hence,
$f(x)=0,\forall x\in E$.

This completes the proof.

\begin{lemma}\label{lem:3.8} Let $(\Omega,\mathcal {F}, P)$ be a
essentially purely $P-$atomic probability space and $(E,\mathcal {P})$
a random locally convex module over $K$ with the base
$(\Omega,\mathcal {F}, P)$. Then $f\in E^{\ast}_{min}$ iff there are
finite $P-$atoms $\{A_{i}~|~1\leq i\leq n\}$,
$\{\xi_{i}~|~1\leq i\leq n\}$ in $L^{0}_+$ and
$\{\mathcal{Q}_{i}~|~1\leq i\leq n\}$ in $\mathcal{F}(\mathcal{P})$
such that
$$|f(x)|\leq\Sigma_{i=1}^{n}\tilde{I}_{A_{i}}\cdot\xi_{i}\cdot\|x\|_{\mathcal{Q}_{i}},\forall x\in E.$$
\end{lemma}

\begin{proof} (1) Necessity: since $(\Omega,\mathcal {F},P)$
is essentially purely $P-$atomic and $E^{\ast}_{min}\subset
E^{\ast}_{\varepsilon,\lambda}$, there exist $\{\xi_{i}~|~i\in
N\}\subset L_{+}^{0}$ and $\{\mathcal {Q}_{i}~|~i\in N\}\subset
\mathcal {F}(\mathcal {P})$ such that
$$|f(x)|\leq\Sigma_{i=1}^{\infty}\tilde{I}_{A_{i}}\cdot\xi_{i}\cdot\|x\|_{\mathcal{Q}_{i}},$$ where $\mathcal{F}$ is generated by at most countably many disjoint atoms
$\{A_{i}~|~i\in N\}$.

We will prove the following claim: there is $n\in N$ such that
$\xi_{i}=0$ on $A_{i}, \forall i\geq n$. Otherwise, there exist $\{i_{k}~|~k\in N\}\subset
N$, $\{x_{k}\in E~|~k\in N\}$ and $\{\mathcal {Q}_{i_{k}}~|~k\in
N\}\subset \mathcal{F}(\mathcal{P})$ such that on $A_{i_{k}}$
$$\|\tilde{I}_{A_{i_{k}}}\cdot x_{k}\|_{\mathcal {Q}_{i_{k}}}>0$$ and $$\tilde{I}_{A_{i_{k}}}\cdot|f(x_{k})|>0.$$ Taking $y_{k}=\tilde{I}_{A_{i_{k}}}\cdot
Q(|f(x_{k})|)\cdot x_{k}$, since $A_{i_{k}}$ is a $P-$atom of
$\mathcal{F}$ for each $k\in N$ and
$P(\Sigma_{k=1}^{\infty}A_{i_{k}})<1$, then $\{P(A_{i_{k}})~|~k\in
N\}\rightarrow 0$ and $\{y_{k}~|~k\in N\}$ converges to 0
under $\mathcal {T}_{\varepsilon,\lambda}$. But
$\{f(y_{x_{k}})~|~k\in N\}$ does not converge to $0$ under $\mathcal
{T}_{c}$ by $|f(y_{k_{i}})|=1$ on $A_{i_{k}}$, which is a
contradiction to $f\in E^{\ast}_{min}$.

(2) Sufficiency: if there are finite $P-$atoms $\{A_{i}~|~1\leq i\leq
n\}$, $\{\xi_{i}~|~1\leq i\leq n\}$ in $L^{0}_+$ and
$\{\mathcal{Q}_{i}~|~1\leq i\leq n\}\subset\mathcal{F}(\mathcal{P})$
such that
$$|f(x)|\leq\Sigma_{i=1}^{n}\tilde{I}_{A_{i}}\cdot\xi_{i}\cdot\|x\|_{\mathcal{Q}_{i}},$$
$\forall x\in E$. Taking $A=\Sigma_{i=1}^{n}A_{i}$,
$E_{A}=\tilde{I}_{A}\cdot E$ and $f_{A}=$ the limitation of $f$ on
$E_{A}$, since $A\cap\mathcal{F}$ is a $\sigma-$algebra generated by
finite $P-$atoms, the $(\varepsilon,\lambda)-$topology is
equivalent to the locally $L^{0}-$convex topology for $E_{A}$. Hence
$f_{A}:(E_{A},\mathcal
{T}_{\varepsilon,\lambda})\rightarrow(\tilde{I}_{A}\cdot
L^{0}(\mathcal{F},K),\mathcal {T}_{c})$ is a continuous
homomorphism, where the base of both $E_{A}$ and $\tilde{I}_{A}\cdot
L^{0}(\mathcal{F},K)$ is taken to be $(A,A\cap {\cal F},P_A)$ with $P_A(A\cap F)=\frac{P(A\cap F)}{P(A)},\forall F\in{\cal F}$. Finally, since $f=0$ on $\tilde{I}_{A^{c}}\cdot E$, it is
clear that $f$ is a continuous
homomorphism from $(E,\mathcal{T}_{\varepsilon,\lambda})$ to $(L^{0}(\mathcal{F},K),\mathcal {T}_{c})$ and $f\in E^{\ast}_{min}$.

This completes the proof.
\end{proof}

\begin{corollary} Let $(E,\mathcal {P})$ be a
random locally convex module over $K$ with base $(\Omega,\mathcal
{F}, P)$ such that $\mathcal{F}$ has at least a $P-$atom. Then $f\in
E^{\ast}_{min}$ iff there are finite $P-$atoms $\{A_{i}~|~1\leq i\leq
n\}$, $\{\xi_{i}~|~1\leq i\leq n\}$ in
$L^{0}_+$ and $\{\mathcal{Q}_{i}~|~1\leq i\leq n\}$ in
$\mathcal{F}(\mathcal{P})$ such that
$$|f(x)|\leq\Sigma_{i=1}^{n}\tilde{I}_{A_{i}}\cdot\xi_{i}\cdot\|x\|_{\mathcal{Q}_{i}},$$
$\forall x\in E$.

\end{corollary}

\begin{proof} It follows immediately from Theorem \ref{thm:minel} and
Lemma \ref{lem:3.8}.

This completes the proof.

\end{proof}


\begin{thebibliography}{00}

\bibitem{Guo-Relations between}Guo, T. X., Relations between some basic results derived
from two kinds of topologies for a random locally convex module, J.
Funct. Anal., 2010, 258: 3024--3047.

\bibitem{Guo-progress}Guo, T. X.,  Recent progress in random metric theory and its
applications to conditional risk measures, Sci. China Ser. A, 2011, doi: 10.1007/s11425-011-4189-6(see also: arXiv:1006.0697v16).

\bibitem{Guo-Survey}Guo, T. X., Survey of recent developments of random metric
theory and its applications in China (II), Acta Anal. Funct.
Appl., 2001, {\bf 3}: 208--230.

\bibitem{FKV} Filipovi\'c, D., Kupper, M. and Vogelpoth, N., Separation and
duality in locally $L^{0}-$convex modules, J Funct Anal, 2009, 256:
3996--4029.

\bibitem{Guo-Peng}Guo T X, Peng S L. A
characterization for an $L(\mu,K)-$topological module to admit
enough canonical module homomorphisms. J Math Anal Appl, 2001,
263: 580--599

\bibitem{Guo-Zhu}Guo, T. X., Zhu, L. H., A characterization of continuous
module homomorphisms on random seminormed modules and its
applications, Acta Math. Sinica English Ser., 2003, 19: 201--208.

\bibitem{Dunford-Schwartz}Dunford, N. and Schwartz, J.T., Linear Operators \uppercase\expandafter{\romannumeral
1}, New York: Interscience, 1957.

\bibitem{G-C}Guo, T. X. and Chen, X. X., Random duality, Sci. China
Ser. A, 2009, 52: 2084--2098.

\bibitem{Guo-Module}Guo, T. X., Module homomorphisms on random normed modules,
Chin. Northeast Math. J., 1996, 12: 102--114.





\bibitem{Conway}Conway, J. B., A Course in Functional Analysis, New York:
Springer-Verlag, 1990.

\bibitem{Dudley} Dudley, R. M., Real Analysis and Probability, England:
Cambridge University Press, 2003.



\end{thebibliography}
\end{document}